\documentclass[reqno,11pt]{amsart}
\usepackage{amssymb}
\usepackage{enumitem}
\setlist[itemize]{leftmargin=2.5em}
\usepackage[T1]{fontenc}
\usepackage{setspace}
\usepackage{stix2,microtype}

\DeclareSymbolFont{mdcharterletters}{OML}{mdbch}{m}{it}
\SetSymbolFont{mdcharterletters}{bold}{OML}{mdbch}{b}{it}
\DeclareMathSymbol{\ell}{\mathord}{mdcharterletters}{"60}
\setbox0=\hbox{$x$}
\usepackage{mathtools,amsthm,thmtools,thm-restate}
\usepackage[noadjust,nocompress]{cite}

\usepackage{hyperref}
\hypersetup{
	colorlinks=true,
	allcolors=blue,
	pdftitle={On the solution to the Erd\H os--Hajnal problem on high-girth high-chromatic subgraphs},
	pdfauthor={Tung Nguyen and Bartosz Walczak}
}
\usepackage[capitalise]{cleveref}

\makeatletter
\def\pdf@Cref#1{%
	\ifcsname r@#1@cref\endcsname
	\expandafter\expandafter\expandafter\pdf@Cref@aux
	\csname r@#1@cref\endcsname\pdf@Cref@end
	\else
	??%
	\fi
}
\def\pdf@Cref@aux#1#2#3\pdf@Cref@end{%
	\pdf@Cref@parse#1\@nil
}
\def\pdf@Cref@parse[#1][#2][#3]#4\@nil{%
	\csname Cref@#1@name\endcsname\space#4%
}
\def\pdf@cref#1{%
	\ifcsname r@#1@cref\endcsname
	\expandafter\expandafter\expandafter\pdf@cref@aux
	\csname r@#1@cref\endcsname\pdf@cref@end
	\else
	??%
	\fi
}
\def\pdf@cref@aux#1#2#3\pdf@cref@end{%
	\pdf@cref@parse#1\@nil
}
\def\pdf@cref@parse[#1][#2][#3]#4\@nil{%
	\csname cref@#1@name\endcsname\space#4%
}
\makeatother

\usepackage{geometry}
\newtheorem{theorem}{Theorem}[section]
\newtheorem{lemma}[theorem]{Lemma}
\newtheorem{problem}[theorem]{Problem}
\Crefname{problem}{Problem}{Problems}
\newtheorem*{claim}{Claim}
\Crefname{subsection}{Subsection}{Subsections}
\Crefname{appendixsection}{Appendix}{Appendices}
\allowdisplaybreaks

\newcommand{\blackqed}{\leavevmode\hbox to .77777em{\hfil\vrule width .6em height .6em depth 0em\hfil}}

\expandafter\let\expandafter\oldproof\csname\string\proof\endcsname
\let\oldendproof\endproof
\renewenvironment{proof}[1][\proofname]{\oldproof[\normalfont\bfseries #1]}{\oldendproof}

\newenvironment{subproof}[1][\proofname]{\oldproof[#1]}{\oldendproof}

\newcommand{\mac}[1]{\text{\usefont{OMS}{cmsy}{m}{n}#1}\mskip 1mu}
\newcommand{\mab}{\mathbb}
\newcommand{\eps}{\varepsilon}
\newcommand{\chis}{\chi^*}
\renewcommand{\subset}{\subseteq}

\renewcommand{\vec}{\overrightarrow}
\newcommand{\edge}{\operatorname{\mathsf{e}}}
\DeclarePairedDelimiter\abs{\lvert}{\rvert}
\DeclarePairedDelimiter\ceil{\lceil}{\rceil}
\DeclarePairedDelimiter\floor{\lfloor}{\rfloor}

\DeclareSymbolFont{timesletters}{OML}{ztmcm}{m}{it}

\makeatletter
\newcommand*\greekrange[3]{%
	\@tempcnta=#2\relax
	\@for\g:=#3\do{%
		\expandafter\DeclareMathSymbol\csname\g\endcsname
		{\mathalpha}{#1}{\the\@tempcnta}%
		\advance\@tempcnta\@ne}}
\makeatother

\greekrange{timesletters}{11}{alpha,beta,gamma,delta,epsilon,zeta,eta,theta,iota,kappa,lambda,mu,nu,xi,pi,rho,sigma,tau,upsilon,phi,chi,psi,omega}
\greekrange{timesletters}{34}{varepsilon,vartheta,varpi,varrho,varsigma,varphi}

\newcommand{\compactsubs}[1]{%
	\fontdimen16\textfont2=#1\fontdimen16\textfont2
	\fontdimen17\textfont2=#1\fontdimen17\textfont2
	\fontdimen16\scriptfont2=#1\fontdimen16\scriptfont2
	\fontdimen17\scriptfont2=#1\fontdimen17\scriptfont2
	\fontdimen16\scriptscriptfont2=#1\fontdimen16\scriptscriptfont2
	\fontdimen17\scriptscriptfont2=#1\fontdimen17\scriptscriptfont2
}
\compactsubs{.8}

\AtBeginDocument{\belowdisplayshortskip\belowdisplayskip}

\makeatletter
\let\old@setaddresses\@setaddresses
\def\@setaddresses{\bigskip{\parindent 0pt\let\scshape\relax\old@setaddresses}}
\makeatother

\begin{document}
	\title[High-girth high-chromatic subgraphs]{
		On the solution to the Erd\H os--Hajnal problem on~high-girth high-chromatic subgraphs
	}
	\author{Tung Nguyen}
	\address{Mathematical Institute and Christ Church, University of Oxford, Oxford, UK}
	\email{\href{mailto:nguyent@maths.ox.ac.uk}{nguyent@maths.ox.ac.uk}}
	\thanks{The first author is supported by a Titchmarsh Research Fellowship and a Christ Church Research Centre Grant.}
	\author{Bartosz Walczak}
	\address{Department of Theoretical Computer Science, Faculty of Mathematics and Computer Science,
		Jagiellonian University, Krak\'ow, Poland}
	\email{\href{mailto:bartosz.walczak@uj.edu.pl}{bartosz.walczak@uj.edu.pl}}
	\thanks{The second author is partially supported by the National Science Centre of Poland grant 2019/34/E/ST6/00443.}
	\begin{abstract}
		A well-known problem of Erd\H os and Hajnal from the 1960s asks whether every graph with huge chromatic number contains a subgraph with large girth and large chromatic number.
		Very recently, Kohlmeyer and Kruer provided a strong negative solution to this problem: a construction of triangle-free graphs with arbitrarily large chromatic number whose subgraphs with no four-cycle have chromatic number at most $6$.
		The purpose of this exposition is to explain the construction method, relate it to relevant literature, and optimise the bound `$6$' to `$3$'.
	\end{abstract}
	\maketitle
	
	\section{Introduction}
	
	Unless stated otherwise, all graphs in this note are finite and simple.
	The {\em girth} of a graph $G$ is the minimum length of a cycle in $G$ or infinity when $G$ is acyclic.
	A {\em colouring} of $G$ is an assignment that gives a colour to each vertex of $G$ so that no two adjacent vertices get the same colour.
	If there is such a colouring using $k$ colours, then we say that $G$ is {\em $k$-colourable}.
	The {\em chromatic number} of $G$, denoted by $\chi(G)$, is the least $k$ such that $G$ is $k$-colourable.
	A classical result of Erd\H os~\cite{Erd59} asserts that there are graphs with arbitrarily large girth and large chromatic number.
	This motivated the following well-known problem of Erd\H os and Hajnal from the 1960s, which was stated by Erd\H os in numerous problem papers \cite{Erd69,Erd71,Erd75,Erd76a,Erd76b,Erd78,Erd79a,Erd79b,Erd81a,Erd81b,Erd82,Erd84,Erd85,Erd90,Erd95,EH85} and is also known more recently as \href{https://www.erdosproblems.com/108}{Erd\H os~problem~108}.%
	\footnote{Erd\H os was often stating an infinitary version of this problem, asking whether for every $g\ge4$, every graph with infinite chromatic number contains a subgraph with girth at least $g$ and infinite chromatic number.
	This formulation is equivalent to the one in \cref{prob:eh} by a standard compactness argument.}
	
	\begin{problem}[Erd\H os and Hajnal]
		\label{prob:eh}
		Is it true that for every pair of integers $g\ge4$ and $k\ge1$, there exists an integer $\ell$ such that every graph with chromatic number greater than $\ell$ contains a subgraph with girth at least $g$ and chromatic number greater than $k$?
	\end{problem}
	
	It is convenient to introduce the following notation: for any integers $g\ge4$ and $k\ge1$, let $\ell(g,k)$ denote the minimum value of $\ell$ with the property asserted in \cref{prob:eh} (equivalently, the maximum chromatic number of a graph in which every subgraph with girth at least $g$ has chromatic number at most $k$), where we set $\ell(g,k)=\infty$ if no such integer $\ell$ exists.
	\Cref{prob:eh} can thus be phrased as follows.
	
	\addtocounter{problem}{-1}
	\begin{problem}[Erd\H os and Hajnal, rephrased]
		Is $\ell(g,k)$ finite for all integers $g\ge4$ and $k\ge1$?
	\end{problem}
	
	Over the years, there have been a number of partial results towards a positive answer to this problem.
	Erd\H os and Hajnal \cite[Theorem~7.7]{EH66} in the 1960s showed that $\ell(g,2)=2\floor{g/2}$ for all $g\ge4$.
	By an elegant argument, R\"odl~\cite{Rod77} in the 1970s proved that $\ell(4,k)$ is finite for all $k\ge1$.
	The problem was also affirmatively resolved for various graph classes: the shift graphs by Tardos and Walczak~\cite{Tar18} (see \cite[Chapter~4]{Apa22} for a proof); the Kneser graphs by Mohar and Wu~\cite{MW23}; the Burling graphs by Pettie, Tardos, and Walczak~\cite{PTW26}; and the Mycielski graphs, via the folklore observation (documented in~\cite{CGL06}) that they contain all triangle-free graphs as induced subgraphs.
	More recently, Steiner~\cite{Ste26} proved that \cref{prob:eh} has a positive answer for all admissible values of $g$ and $k$ if `girth' is relaxed to `odd girth', which is defined with `odd cycle' replacing `cycle' in the definition of girth; he also significantly improved on R\"odl's upper bound on $\ell(4,k)$.
	
	In an approach to resolve \cref{prob:eh} in the negative, Pettie, Tardos, and Walczak~\cite{PTW26}, by studying the class of Burling graphs, proved a lower bound on $\ell(5,k)$ that is a tower of twos of height linear in $k$, which hinted at the possibility that $\ell(5,k)$ might be infinite for all sufficiently large $k$.
	
	Very recently, \cref{prob:eh} has been resolved by Kohlmeyer and Kruer~\cite{KK26a,KK26b} with substantial assistance from AI tools; they proved that $\ell(5,k)$ is infinite for all $k\ge6$.
	The resolution constructs triangle-free graphs with arbitrarily large chromatic number whose $C_4$-free subgraphs (i.e., subgraphs with no four-cycle) are all $6$-colourable.
	
	\begin{theorem}[{\cite[Theorem~1.1]{KK26a,KK26b}}]
		\label{thm:girthchi}
		For every $k\ge1$, there exists a triangle-free graph with chromatic number greater than $k$ whose $C_4$-free subgraphs are all $6$-colourable.
	\end{theorem}
	
	The construction behind \cref{thm:girthchi} has been thus far described in two manuscripts: \cite{KK26a}, produced with substantial assistance from AI tools (as acknowledged by the authors) and published on Kohlmeyer's webpage; and \cite{KK26b}, AI-generated from the Lean code and published by the platform Conjectures.io.
	It is the latter that achieved some reach among the mathematical community, because of public announcements from the platform's maintainers.
	Especially the presentation in~\cite{KK26b} may not be readily accessible to researchers in graph theory; furthermore, both manuscripts lack any references to the literature except for \cite{KK26a} citing the work of Janzer, Steiner, and Sudakov~\cite{JSS26} on which the construction builds.
	The purpose of this note is thus twofold:
	\begin{itemize}
		\item firstly, to present the ideas behind \cref{thm:girthchi} in a manner more accessible and less numerically forbidding; and
		\item secondly, to connect every single key step of the construction to the relevant literature that we are aware of.
	\end{itemize}
	
	Pursuing these two objectives also led us to the following improvement of \cref{thm:girthchi} that uses essentially the same proof method.
	
	\begin{theorem}
		\label{thm:main}
		For every $k\ge1$, there exists a triangle-free graph with chromatic number greater than $k$ whose $C_4$-free subgraphs are all $3$-colourable.
	\end{theorem}
	
	In particular, this implies that $\ell(5,k)$ is infinite for all $k\ge3$.
	This, in turn, gives a full characterisation of the finiteness of $\ell(g,k)$ for all $g\ge4$ and $k\ge1$, as follows.
	
	\begin{theorem}
		\label{thm:exist}
		$\ell(g,k)$ is finite if and only if $g=4$ or $k\in\{1,2\}$.
	\end{theorem}
	
	After we released the first version of this note, Steiner~\cite{Ste-personal} shared with us an interesting corollary to the construction behind \cref{thm:girthchi,thm:main}.
	A graph is {\em $H$-free} if it contains no subgraph isomorphic to $H$.
	Let $K_{s,t}$ denote the complete bipartite graph with one part of size $s$ and the other part of size $t$ (then $C_4=K_{2,2}$).
	The following straightforward extension of \cref{thm:main} is needed.
	
	\begin{theorem}
		\label{thm:biclique-free}
		For all integers $s\ge t\ge2$ and $k\ge1$, there exists a triangle-free graph with chromatic number greater than $k$ whose $K_{s,t}$-free subgraphs are all $(t+1)$-colourable.
	\end{theorem}
	
	Here is a concept proposed by Steiner~\cite{Ste-personal}: a graph $H$ is {\em $\chi$-avoidable} if for every $k\ge1$, every graph with sufficiently large chromatic number contains an $H$-free subgraph with chromatic number greater than $k$.%
	\footnote{The rationale behind this term is that a $\chi$-avoidable graph can always be avoided in some high-chromatic subgraph.}
	\Cref{thm:girthchi} thus implies that $C_4$ is not $\chi$-avoidable, while the aforementioned results of R\"odl~\cite{Rod77} and Steiner~\cite{Ste26} imply that cycles of odd length are $\chi$-avoidable.
	Since every bipartite graph is a subgraph of $K_{s,t}$ for some $s,t\geq 2$ while every non-bipartite graph contains a cycle of odd length, these results along with \Cref{thm:biclique-free} give a full characterisation of $\chi$-avoidable graphs.
	
	\begin{theorem}[Steiner~\cite{Ste-personal}]
		\label{thm:chi-avoidable}
		A graph is $\chi$-avoidable if and only if it is non-bipartite.
	\end{theorem}
	
	The rest of this note is devoted to the proof of \cref{thm:main}; the proof of \cref{thm:biclique-free} will be a straightforward generalisation thereof.
	
	\section{Key lemma}
	\label{sec:prelim}
	
	In this section, we state the key lemma behind the proof of \cref{thm:main}.
	To do so, we need a few definitions.
	
	Let $V(G)$ and $E(G)$ denote the vertex set and the edge set of a graph $G$, respectively.
	For an integer $k\ge0$, a graph $G$ is {\em $k$-degenerate} if $V(G)$ admits an ordering $v_1,\ldots,v_n$ such that every $v_i$ has at most $k$ neighbours in $\{v_j:j>i\}$.
	Thus, $G$ is $k$-degenerate if and only if all subgraphs of $G$ have minimum degree at most $k$, and if $G$ is $k$-degenerate then $\chi(G)\le k+1$, as the vertices of $G$ can be coloured greedily in the order $v_n,\ldots,v_1$.
	A {\em stable set} in a graph $G$ is a set of vertices no two of which are adjacent.
	The {\em fractional chromatic number} of $G$, denoted by $\chis(G)$, is the minimal $r\ge0$ such that there exists $f\colon\mac I\to \mab R_{\ge0}$ satisfying $r=\sum_{S\in\mac I}f(S)$ and $\sum_{S\ni v}f(S)\ge1$ for all $v\in V(G)$, where $\mac I$ is the family of all stable sets of $G$.
	Via linear programming duality, $\chis(G)$ is the maximal $r\ge0$ such that there exists $w\colon V(G)\to\mab R_{\ge0}$ satisfying $r=\sum_{v\in V(G)}w(v)$ and $\sum_{v\in S}w(v)\le1$ for all $S\in\mac I$.
	In this note, it is convenient to use the latter as the definition of $\chis(G)$.
	It follows that $\chis(G)\le\chi(G)$ for every graph $G$; indeed, every colour class in a colouring of $G$ with $\chi(G)$ colours is a stable set, which implies that $\sum_{v\in V(G)}w(v)\le\chi(G)$ for every $w\colon V(G)\to\mab R_{\ge0}$ satisfying $\sum_{v\in S}w(v)\le1$ for all $S\in\mac I$.
	
	Given these definitions, the key lemma in the proof of \cref{thm:main} is the following standalone result, which is a strengthening of \cite[Theorem~2.1]{KK26a} and which may be of independent interest.
	
	\begin{restatable}{theorem}{thmorient}
		\label{thm:orient}
		For every pair of integers $q,\Delta\ge1$, there exists a graph $G$ satisfying the following:
		\begin{itemize}
			\item $G$ is $\floor{4608q^2\ln(4q)}$-degenerate;
			\item $\chis(G)>q$; and
			\item all subgraphs of $G$ with maximum degree at most $\Delta$ are $2$-degenerate.
		\end{itemize}
	\end{restatable}
	
	In particular, when $\Delta=3$, this result implies that there are graphs $G$ with arbitrarily large fractional chromatic number and no $3$-regular subgraphs.
	
	The proof of \cref{thm:orient} is an extension of the argument in \cite[Section~3]{KK26a} and will be completed in \cref{sec:2degen}.
	To explain the main ideas behind the proof and connect them to relevant literature, in \cref{sec:random} we will explain the proof of a weaker statement, which is sufficient for the proof of \cref{thm:main} with `$3$' replaced by `$4$'.
	Before that, in \cref{sec:shift}, we will show the derivation of \cref{thm:main} from \cref{thm:orient} via the construction of {\em line digraphs}.
	
	\section{Line digraphs}
	\label{sec:shift}
	
	This section deduces \cref{thm:main} from \cref{thm:orient}, following \cite[Sections 2 and~4]{KK26a} and \cite[Sections 2, 3, and~6]{KK26b}.
	In what follows, let $[k]:=\{1,\ldots,k\}$ for every integer $k\ge0$.
	For a graph $G$ and a set $X\subset V(G)$, let $G[X]$ denote the subgraph of $G$ induced by $X$.
	
	A {\em digraph} (also known as {\em directed graph}) $\vec G$ consists of a set of vertices, denoted by $V(\vec G)$, and a set of edges that are ordered pairs of vertices, denoted by $E(\vec G)$.
	In this note, we disallow loops in a digraph.
	The {\em underlying graph} of a digraph $\vec G$ is a simple graph $G$ on the same vertex set such that two vertices are adjacent in $G$ if and only if there is an edge in either direction between them in $\vec G$.
	For simplicity of notation, we will always put an arrow on top of a letter denoting a digraph, and we will use the same letter without an arrow to denote its underlying graph, as in the preceding sentence.
	We say that $\vec G$ is {\em acyclic} if it has no directed cycle; thus $\vec G$ is acyclic if and only if $V(\vec G)$ admits an ordering $v_1,\ldots,v_n$ such that all edges are directed from a vertex with a smaller index to a vertex with a larger index; equivalently, each $v_i$ has no in-neighbour in $\{v_j:j>i\}$.
    
	The main objects in this section are {\em line digraphs}, introduced by Harary and Norman~\cite{HN60} and defined as follows.
	The {\em line digraph} $\vec L(\vec G)$ of a digraph $\vec G$ is the digraph with vertex set $E(\vec G)$ where there is a directed edge from $(u,v)$ to $(x,y)$ if and only if $v=x$.
	(In~\cite{KK26a,KK26b}, line digraphs are called {\em arc graphs} when $\vec G$ is acyclic.)
	We use the notation $L(\vec G)$ for the underlying graph of the line digraph $\vec L(\vec G)$.
	A colouring of $L(\vec G)$ can be naturally viewed as a colouring of the edges of $\vec G$ such that no two edges $(u,v)$ and $(x,y)$ with $v=x$ get the same colour; such colourings are known as \emph{arc colourings} and were introduced by Harner and Entringer~\cite{HE72}.
	Perhaps the best-known instance of graphs of the form $L(\vec G)$ is when the base digraph $\vec G$ is a transitive tournament; they are known as {\em shift graphs} and were introduced by Erd\H os and Hajnal~\cite{EH66} as an example of triangle-free graphs with large chromatic number.
	It follows directly from the definition that the graphs of the form $L(\vec G)$ for acyclic digraphs $\vec G$ are exactly the induced subgraphs of shift graphs.
	On a historical note, shift graphs were suggested by Erd\H os and R\"odl as a counterexample candidate for \cref{prob:eh} (see~\cite{Erd78}), but they turned out to yield a positive instance~\cite{Apa22,Tar18}.
	Recently, Adenwalla, Braunfeld, Hons, Sylvester, and Zamaraev \cite[Problem~1.13]{ABHSZ26} asked whether all induced subgraphs of shift graphs are positive instances of \cref{prob:eh}.
	As it turns out, it is these graphs where the counterexamples to \cref{prob:eh} have been found.
	
	First observe that if $\vec G$ is acyclic, then $L(\vec G)$ is triangle-free, as the only way for a triangle to arise in $L(\vec G)$ is for three edges of $\vec G$ to form a directed triangle.
	Recall also the following well-known observation.
	
	\begin{lemma}[{\cite[Theorem~9]{HE72}}]
		\label{lem:shiftchi}
		$\chi(G)\le2^{\chi(L(\vec G))}$ for all digraphs $\vec G$.
	\end{lemma}
	
	\begin{proof}
		Consider a colouring of $L(\vec G)$ with $\chi(L(\vec G))$ colours; and for every $v\in V(\vec G)$, let $S_v$ be the set of colours assigned to the edges of $\vec G$ going out of $v$.
		Observe that if $(u,v)\in E(\vec G)$, then the colour of $(u,v)$ appears in $S_u$ but cannot appear in $S_v$, for otherwise there would be some $(v,w)\in E(\vec G)$ having that colour while $(u,v)$ and $(v,w)$ form an edge in $L(\vec G)$, contradicting the validity of the colouring.
		This shows that the assignment $v\mapsto S_v$ is a colouring of $G$ with at most $2^{\chi(L(\vec G))}$ colours, proving \cref{lem:shiftchi}.
	\end{proof}
	
	The next step studies the $C_4$-free subgraphs of $L(\vec G)$ when $\vec G$ is acyclic and has bounded maximum out-degree.
	For clarity (and to derive \cref{thm:biclique-free}), we will consider the more general case of $K_{s,t}$-free subgraphs of $L(\vec G)$ for $s,t\ge2$.
	The approach in this step is similar to the analysis of Sadhukhan \cite[Theorem~11]{Sad25} on the chromatic number of $K_{s,t}$-free induced subgraphs of shift graphs.
	A subdigraph $\vec H$ of $\vec L(\vec G)$ such that $H$ is $K_{s,t}$-free gives rise to a subdigraph $\vec Q$ of $\vec G$ as follows: take the vertices of $\vec H$ with out-degree at least $t$ in $\vec H$, and make them the edges of the subdigraph $\vec Q$.
	The following lemma (restricted to acyclic digraphs $\vec G$ for clarity) says that $Q$ has bounded maximum degree (depending on the maximum out-degree of $\vec G$), and it is the only place in the proof of \cref{thm:main} that uses the $C_4$-free hypothesis.
	Its case $s=t=2$ is implicit in \cite[Lemma~2.4]{KK26a} and explicit in \cite[Lemma~3.1]{KK26b}.
	
	\begin{lemma}
		\label{lem:q}
		Let $s,t\ge2$ and $d\ge1$ be integers, let $\vec G$ be an acyclic digraph with maximum out-degree at most $d$, let $H$ be a $K_{s,t}$-free subgraph of $L(\vec G)$, and let $\vec H$ be the corresponding subdigraph of $\vec L(\vec G)$.
		Let $\vec Q$ be the subdigraph of $\vec G$ with vertex set $V(\vec G)$ and edge set
        \[\left\{\,\vec e\in V(\vec H):\vec e\text{ has out-degree at least }t\text{ in }\vec H\,\right\}.\]
		Then $Q$ has maximum degree at most $d+(s-1)\binom dt$.
	\end{lemma}
	
	\begin{proof}
		Let $v\in V(\vec G)$, and let $N$ be the set of out-neighbours of $v$ in $\vec G$; then $\abs N\le d$.
		For each edge $(u,v)\in E(\vec Q)$ going into $v$ in $\vec Q$, the definition of $\vec Q$ gives $t$ edges $(v,w_1),\ldots,(v,w_t)\in E(\vec G)$ that are out-neighbours of $(u,v)$ in $\vec H$; in particular $w_1,\ldots,w_t\in N$.
		On the other hand, since $H$ is $K_{s,t}$-free, for any choice of distinct $w_1,\ldots,w_t\in N$, there are at most $s-1$ elements $(u,v)\in V(\vec H)$ such that all $(v,w_1),\ldots,(v,w_t)$ are out-neighbours of $(u,v)$ in $\vec H$.
		It follows that $v$ has at most $(s-1)\binom{\abs N}t\le(s-1)\binom dt$ in-neighbours in $\vec Q$, and so $v$ has degree at most $d+(s-1)\binom dt$ in $Q$.
		This proves \cref{lem:q}.
	\end{proof}
	
	The next lemma is the first main difference between the arguments presented here and in~\cite{KK26a}.
	To explain this using the notation from \cref{lem:q}, when $s=t=2$ and $\vec G$ is acyclic, the proof of \cite[Lemma~2.4]{KK26a} and \cite[Lemma~3.2]{KK26b} observes that $\chi(H[E(\vec Q)])\le\chi(Q)$ and the graph $H[V(\vec H)\setminus E(\vec Q)]$ is $1$-degenerate and thus $2$-colourable; then it colours $H[E(\vec Q)]$ and $H[V(\vec H)\setminus E(\vec Q)]$ separately to obtain a colouring of $H$ with at most $\chi(Q)+2$ colours.
	The argument in the next lemma colours $H$ in a more efficient way; it is a simple extension of the argument behind \cite[Theorem~8]{HE72}.
	
	\begin{lemma}
		\label{lem:extend}
		Let $r\ge t\ge1$ be integers.
		Let $\vec G$ be an acyclic digraph, and let $\vec H$ be a subdigraph of $\vec L(\vec G)$.
		Let $\vec Q$ be any subdigraph of $\vec G$ such that $V(\vec Q)=V(\vec G)$ and $E(\vec Q)$ contains all edges of $\vec G$ which, considered as vertices in $\vec H$, have out-degree at least $t$ in $\vec H$.
		If $Q$ is $\binom rt$-colourable, then $H$ is $r$-colourable.
	\end{lemma}
	
	\begin{proof}
		Since $Q$ is $\binom rt$-colourable, it admits a colouring $f$ that maps $V(Q)=V(\vec G)$ to the $t$-element subsets of $[r]$.
		We show that $H$ has a colouring into $[r]$ such that every $(u,v)\in V(H)$ gets a colour from the set $f(u)$.
		
		We colour $H$ in an order ensuring that when we assign a colour to $(u,v)\in V(H)$, all the out-neighbours of $(u,v)$ in $H$ have already been coloured; such an order exists because $\vec G$ is acyclic and so $\vec L(\vec G)$ is acyclic as well.
		Furthermore, we always ensure that the colour of $(u,v)\in V(H)$ belongs to the set $f(u)$.
		To see that we can always find for $(u,v)$ an eligible colour in $f(u)$, observe that
		\begin{itemize}
			\item if $(u,v)\in E(\vec Q)$, then a colour for $(u,v)$ can be chosen from the set $f(u)-f(v)$, which is nonempty because $f(u)\neq f(v)$ (for $f$ is a colouring of $Q$) and both sets have the same size $t$ (note also that every already-coloured out-neighbour $(v,w)$ of $(u,v)$ received a colour in $f(v)$);
			\item if $(u,v)\notin E(\vec Q)$, then $(u,v)$ has fewer than $t$ out-neighbours in $\vec H$, so there is a colour in $f(u)$ that is not used by any out-neighbour of $(u,v)$ in $\vec H$.
		\end{itemize}
		This proves \cref{lem:extend}.
	\end{proof}
	
	We are now ready to prove \cref{thm:main} assuming \cref{thm:orient}.
	
	\begin{proof}[Proof of \cref{thm:main} assuming \cref{thm:orient}]
		Let $q:=2^k$, $d:=\floor{4608q^2\ln(4q)}$, and $\Delta:=d+\binom d2$.
		Let $G$ be the graph given by \cref{thm:orient}.
		Since $G$ is $d$-degenerate, the edges of $G$ can be directed to obtain an acyclic digraph $\vec G$ with maximum out-degree at most $d$.
		We claim that the graph $L(\vec G)$ satisfies the theorem.
		To see this, first note that $L(\vec G)$ is triangle-free (as $\vec G$ is acyclic) and \cref{lem:shiftchi} implies $2^{\chi(L(\vec G))}\ge\chi(G)\ge\chis(G)>q=2^k$; and so $\chi(L(\vec G))>k$.
		Now, let $H$ be a $C_4$-free subgraph of $L(\vec G)$, let $\vec H$ be the corresponding subdigraph of $\vec L(\vec G)$, and let $\vec Q$ be the subdigraph of $\vec G$ with vertex set $V(\vec G)$ whose edge set contains those and only those edges of $\vec G$ which, considered as vertices in $\vec H$, have out-degree at least $2$ in $\vec H$.
		By \cref{lem:q} with $s=t=2$, the graph $Q$ has maximum degree at most $d+\binom d2=\Delta$.
		Then, by the choice of $G$, the graph $Q$ is $2$-degenerate and so has chromatic number at most $3=\binom32$.
		Hence $\chi(H)\le3$ by \cref{lem:extend} with $r=3$ and $t=2$.
		This proves \cref{thm:main}.
	\end{proof}
	
	The proof of \cref{thm:biclique-free} is identical except for a different choice of some parameters.
	Specifically, in the case of $K_{s,t}$-free subgraphs, one can keep the same $q$ and $d$ and define $\Delta:=d+(s-1)\binom dt$.
	By \cref{lem:q}, the subgraph $Q$ of interest has maximum degree at most $\Delta$, and so it has chromatic number at most $3\le\binom{t+1}t$ since $t\ge2$.
	Then \cref{lem:extend} with $r=t+1$ yields $\chi(H)\le t+1$.
	
	
	\section{Low average degree}
	\label{sec:random}
	
	In this section, to illustrate the main ideas behind the proof of \cref{thm:orient}, we follow (with small simplifications) the argument in \cite[Section~3]{KK26a} and \cite[Sections 4 and~5]{KK26b} to present the proof of the following result.
	
	\begin{theorem}[{cf.~\cite[Theorem~3.1]{KK26a}}]
		\label{thm:orient1}
		For every pair of integers $q,\Delta\ge1$, there exists a graph $G$ satisfying the following:
		\begin{itemize}
			\item $G$ is $(16q^2-1)$-degenerate;
			\item $\chis(G)>q$; and
			\item all subgraphs of $G$ with maximum degree at most $\Delta$ have average degree less than $6$.
		\end{itemize}
	\end{theorem}
	
	This is analogous to \cref{thm:orient}, with `are $2$-degenerate' replaced by `have average degree less than $6$' and a slightly strengthened first outcome.
	In the last outcome, average degree less than $6$ implies minimum degree at most $5$, which in turn implies $5$-degeneracy by repeatedly taking out vertices of degree at most $5$ and using the third condition for each remaining subgraph; $5$-degeneracy then implies $6$-colourability.
	The same result with `$6$' replaced by `$4$' is essentially proved in \cite[Theorem~3.1]{KK26a} and implicitly in \cite[Section~5]{KK26b}.
	The choice of `$6$' here is not only for clarity of exposition, but also because \cref{thm:orient1} suffices to prove a variant of \cref{thm:main} with `$3$' replaced by `$4$' (itself already improving on \cref{thm:girthchi}).
	To see the latter, recall from the discussion above that in \cref{thm:orient1}, all subgraphs of $G$ with maximum degree at most $\Delta$ have chromatic number at most $6=\binom42$, so it suffices to adapt the deduction of \cref{thm:main} from \cref{thm:orient} in \cref{sec:shift} (now with $r=4$ and $t=2$); we omit the details.
	
	In what follows, for a simple graph $G$, let $\abs G:=\abs{V(G)}$; and for every $S\subset V(G)$, let $G[S]$ be the subgraph of $G$ induced on $S$.
	The desired graph $G$ will be constructed randomly in the following way.
	For $C\ge2$, consider integers $n_1\ge\ldots\ge n_C\ge1$, and let $G$ be a random multipartite graph with parts $V_1,\ldots,V_C$ defined as follows:
	\begin{itemize}
		\item $\abs{V_i}=n_i$ for all $i\in[C]$; and
		\item independently for all $i,j\in[C]$ with $i<j$ and for each $u\in V_i$, pick a uniformly random vertex $v$ in $V_j$ and make $uv$ an edge of $G$.
	\end{itemize}
	
	This exact random graph model was recently analysed by Janzer, Steiner, and Sudakov~\cite{JSS26}, who used it to construct graphs with large fractional chromatic number and with no $4$-regular subgraphs.
	An analogous model but with edges only between $V_1$ and $V_2\cup\cdots\cup V_C$ (so restricted to $i=1$) was introduced by Pyber, R\"odl, and Szemer\'edi~\cite{PRS95} in the 1990s to construct graphs with superlinear density and no $3$-regular subgraph, and later adapted by Chakraborti, Janzer, Methuku, and Montgomery~\cite{CJMM26} for improved constructions of graphs with high density and no regular subgraphs, and by Dvo\v r\'ak, Ossona de Mendez, and Wu~\cite{DOW20} to construct bipartite graphs of high density with no $1$-subdivisions of graphs with Hall ratio greater than $18$.
	A model very similar to the one in~\cite{JSS26}, such that the edges $uv$ with $u\in V_i$ and $v\in V_j$ for $i<j$ are chosen independently at random with probability $|V_j|^{-1}$, was considered by Steiner~\cite{Ste25}.
	
	As we shall see below, much of the argument in this section (and in \cite[Section~3]{KK26a} and \cite[Section~5]{KK26b}) essentially appeared in the aforesaid work of Janzer, Steiner, and Sudakov~\cite{JSS26}.
	The main difference from~\cite{JSS26} is the change in the maximum degree bound from a universal constant (which is $4$ in~\cite{JSS26}) to the more flexible parameter $\Delta$, which allows for a combination with \cref{lem:q}.
	
	Given the definition of $G$, order all vertices of $V_i$ before all vertices of $V_j$ whenever $i<j$, in an arbitrary order within each part.
	This ordering witnesses that $G$ is $(C-1)$-degenerate, which already confirms the first property in \cref{thm:orient1} with $C=16q^2$.
	Let us next arrange the sequence $(n_i)_{i=1}^C$ so that the second property $\chis(G)>q$ holds with high probability.
	Here, to simplify our exposition, we will first assume an {\em exponential} decay of this sequence, in the sense that
	\begin{equation}
		\label{eq:exp}
		n_{i-1} \ge 2n_i \quad\text{for all }i\in\{2,\ldots,C\},
	\end{equation}
	which implies that
	\[n_i+n_{i+1}+\cdots+n_C \le 2n_i \quad\text{for all }i\in[C].\]
	
	The aim is to show that under the assumption \eqref{eq:exp}, if $C=16q^2$ then $\chis(G)>q$ with high probability, as desired.
	To do so, it may be helpful to view $G$ as an `unbalanced' variant of the sparse Erd\H os--R\'enyi random graphs that Erd\H os~\cite{Erd59} employed to construct graphs with large girth and large chromatic number.
	In this sense, to show that $\chis(G)$ is large, it is natural to assign weight $1/n_i$ to each vertex in $V_i$ for every $i$, and then to show that $G$ has no heavy stable set with high probability.
	The next lemma does exactly that.
	Its proof, following the argument in \cite[Lemma~3.5]{KK26a} and \cite[Subsection~5.1]{KK26b}, is a simplified version of the proof of \cite[Lemma~2.3]{JSS26}; the latter proves $\chis(G)>q$ when $C=\Omega(q\log q)$ but requires a much faster decay of the sequence $(n_i)_{i=1}^C$.
	
	\begin{lemma}[{cf.~\cite[Lemma~2.3]{JSS26}, \cite[Lemma~3.5]{KK26a}, and \cite[Subsection~5.1]{KK26b}}]
		\label{lem:stable}
		Assume \eqref{eq:exp}.
		If $C=16q^2$, then $\chis(G)>q$ with probability greater than $1/2$.
	\end{lemma}
	
	\begin{proof}
		For each $i\in[C]$ and $v\in V_i$, let $w(v):=1/n_i$.
		For each $S\subset V(G)$, let $w(S):=\sum_{v\in S}w(v)$.
		Thus $w(V_i)=1$ for all $i\in[C]$, and $w(V(G))=C$.
		Let $w_{\max}$ be the maximum $w(S)$ over all stable sets $S\subset V(G)$.
		If $w_{\max}<C/q$, then the weight function $w/w_{\max}$ witnesses $\chis(G)\ge C/w_{\max}>q$, by the definition of fractional chromatic number.
		Therefore, it suffices to show that $G$ has a stable set of weight at least $C/q=16q$ with probability less than $1/2$.
		
		To see this, we say that a subset $S\subset V(G)$ is {\em good} if the following hold:
		\begin{itemize}
			\item $w(S)\ge8q$; and
			\item there exists $I\subset[C]$ with $w(S\cap V_j)=0$ for all $j\in[C]\setminus I$ and $w(S\cap V_j)\ge1/(2q)$ for all $j\in I$.
		\end{itemize}
		
		Observe that every $S\subset V(G)$ with $w(S)\ge16q$ contains a good subset.
		Indeed, for $I:=\{j\in[C]:w(S\cap V_j)\ge1/(2q)\}$ we have $\sum_{j\in I}w(S\cap V_j)\ge w(S)-C\cdot 1/(2q)\ge8q$; and so $S':=\bigcup_{j\in I}(S\cap V_j)$ is a good subset of $S$.
		
		To prove the lemma, it thus remains to show that $G$ has a good stable set with probability less than $1/2$.
		To this end, for each $i\in[C]$ let $B_i$ be the set of good $S\subset V(G)$ for which $i$ is the least index satisfying $w(S\cap V_i)\ge1/(2q)$; then $S\subset V_i\cup V_{i+1}\cup\cdots\cup V_C$ for all $S\in B_i$.
		Now, fix $i\in[C]$, and note that $\abs{B_i}\le2^{n_i+n_{i+1}+\cdots+n_C}\le2^{2n_i}$.
		Let $S\in B_i$ with $I=\{j\in[C]:w(S\cap V_j)\ge1/(2q)\}$; then $i$ is the least element in $I$.
		Let $w_j:=w(S\cap V_j)=\abs{S\cap V_j}/n_j\le1$ for all $j\in I$; then $\abs I\ge w(S)\ge8q>1$.
		If $S$ is stable in $G$, then $G$ contains no edge between $S\cap V_i$ and $\bigcup_{j\in I\setminus\{i\}}(S\cap V_j)$; and the latter event occurs with the following probability (note that $w_i\cdot w(S)-w_i^2\ge(2q)^{-1}\cdot 8q-1=3$):
		\[\prod_{j\in I\setminus\{i\}}(1-w_j)^{\abs{S\cap V_i}}
		\le e^{-\abs{S\cap V_i}\sum_{j\in I\setminus\{i\}}w_j}
		= e^{-w_in_i\cdot(w(S)-w_i)}
		= e^{-(w_i\cdot w(S)-w_i^2)n_i}
		\le e^{-3n_i}.\]
		
		Since the above holds for all $i\in[C]$, the union bound implies that $G$ contains a good stable set with probability at most
		\[\sum_{i\in[C]}\abs{B_i}\cdot e^{-3n_i}
		\le \sum_{i\in[C]}2^{2n_i}\cdot e^{-3n_i}
		\le \sum_{i\in[C]}2^{-2n_i}
		< \sum_{n\ge1}2^{-2n} < 1/2,\]
		where the third inequality holds since we are assuming \eqref{eq:exp}.
		This proves \cref{lem:stable}.
	\end{proof}
	
	It remains to show that the third property of \cref{thm:orient1} holds with high probability.
	To do so, we employ the following simple lemma, a stronger version of which is implicit in the proof of \cite[Lemma~3.3]{KK26a} and explicit in \cite[Proposition~4.1]{KK26b} and which may be thought of as a very basic instance of the `absorption method' in extremal combinatorics: removing a negligible-size set of vertices of bounded degree from a graph does not significantly alter the edge density.
	
	\begin{lemma}[{cf.~\cite[Proposition~4.1]{KK26b}}]
		\label{lem:absorb}
		Let $\Delta\ge1$.
		Let $Q$ be a graph with $\abs Q\ge1$ and maximum degree at most $\Delta$.
		Assume that $V(Q)$ admits a partition $(X,Y,Z)$ allowing empty parts, such that:
		\begin{itemize}
			\item $Q[X]$ has at most $3\abs X/2$ edges;
			\item $Y$ is stable in $Q$;
			\item $Q$ has at most $\abs X$ edges between $X$ and $Y$; and
			\item $\abs Z<\abs Q/(2\Delta)$.
		\end{itemize}
		Then $Q$ has average degree less than $6$.
	\end{lemma}
	
	\begin{proof}
		By the first three properties, the number of edges in $Q[X\cup Y]$ is at most $5\abs X/2$.
		The last property implies that $Q$ has at most $\Delta\abs Z<\abs Q/2$ edges with an endpoint in $Z$.
		This implies that $Q$ has fewer than $5\abs X/2+\abs Q/2\le3\abs Q$ edges in total, so it has average degree less than $6$.
		This proves \cref{lem:absorb}.
	\end{proof}
	
	The plan is to apply \cref{lem:absorb} to every subgraph $Q$ of $G$ with $\abs Q\ge1$ and maximum degree at most $\Delta$.
	Given the multipartite setting, a natural choice for $X,Y,Z$ is $X=V(Q)\cap\bigcup_{1\le j<i}V_j$, $Y=V(Q)\cap V_i$, and $Z=V(Q)\cap\bigcup_{i<j\le C}V_j$ for some $i\in[C]$.
	Thus, for notational convenience, define
	\[L_i := \bigcup_{1\le j<i}V_j \quad\text{and}\quad
	R_i := \bigcup_{i<j\le C}V_j \quad\text{for all }i\in[C].\]
	
	It remains to ensure that $Q[X]$ satisfies the first property in \cref{lem:absorb} by analysing the random multipartite graph $G$.
	To carry this out, we want to show that with high probability, all small subgraphs of each left-hand segment $L_i$ have average degree at most $3$.
	The analysis will assume further that the sequence $(n_i)_{i=1}^C$ satisfies the following much sharper {\em double-exponential} decay:
	\begin{equation}
		\label{eq:dexp}
		n_{i-1}^3 \ge (72KC)\cdot n_1^2n_i \quad\text{for all }i\in\{2,\ldots,C\},
	\end{equation}
	where $K\ge1$ is some parameter to be chosen later.
	Under this assumption, the following lemma, which (with a different choice of constants) is implicit in the proof of \cite[Lemma~2.1]{JSS26}, guarantees the desired property with high probability.
	
	\begin{lemma}[{cf.~\cite[Lemma~2.1]{JSS26}, \cite[Lemma~3.4]{KK26a}, and \cite[Subsections 5.2 and~5.3]{KK26b}}]
		\label{lem:sparse}
		Assume \eqref{eq:dexp}.
		Then $G$ satisfies the following with probability greater than $1/2$: for all $i\in\{2,\ldots,C\}$ and all $X\subset L_i$ with $\abs X\le Kn_i$, $G[X]$ has at most $3\abs X/2$ edges.
	\end{lemma}
	
	\begin{proof}
		For each $i\in\{2,\ldots,C\}$ and each integer $r$ with $4\le r\le Kn_i$ (in particular $r\le Kn_i\le n_{i-1}$), let $E_{i,r}$ be the event that there exists $X\subset L_i$ such that $\abs X=r$ and $G[X]$ has at least $3r/2$ edges.
		Every pair of vertices in $L_i$ becomes an edge of $G$ with probability at most $1/n_{i-1}$.
		For every set $P$ of pairs of vertices in $L_i$, the events that the pairs in $P$ become edges of $G$ are independent unless $P$ contains two pairs $uv$ and $uv'$ with $u\in V_j$ and $v,v'\in V_s$ where $1\le j<s<i$; such two pairs cannot be edges of $G$ simultaneously.
		Consequently, the probability that all pairs in $P$ become edges of $G$ is at most $n_{i-1}^{-\abs P}$.
		Since $\abs{L_i}=n_1+\cdots+n_{i-1}\le2n_1$ and by the inequality $\binom nk\le(\frac{3n}k)^k$ for all $n\ge k\ge1$, the union bound over all sets $X\subset L_i$ with $\abs X=r$, and all sets $P$ of $\ceil{3r/2}$ pairs of vertices in $X$, implies that $E_{i,r}$ occurs with probability at most
		\[\begin{aligned}
			\binom{\abs{L_i}}r\cdot\binom{r(r-1)/2}{\ceil{3r/2}}\cdot n_{i-1}^{-\ceil{3r/2}}
			&\le \left(\frac{6n_1}r\right)^r\cdot r^{\ceil{3r/2}}\cdot n_{i-1}^{-\ceil{3r/2}}\\
			&\le \left(\frac{6n_1}r\right)^r\cdot\left(\frac r{n_{i-1}}\right)^{3r/2}
			= \left(\frac{36n_1^2r}{n_{i-1}^3}\right)^{r/2}
			\le \left(\frac{36Kn_1^2n_i}{n_{i-1}^3}\right)^{r/2}
			\le (2C)^{-r/2},
		\end{aligned}\]
		where the last inequality holds since we are assuming \eqref{eq:dexp}.
		Thus, since $L_1=\emptyset$, summing over $i\in\{2,\ldots,C\}$ and $4\le r\le Kn_i$ and applying the union bound imply that the probability that there exist $i\in[C]$ and $X\subset L_i$ such that $\abs X\le Kn_i$ and $G[X]$ has more than $3\abs X/2$ edges is at most
		\[C\cdot\sum_{r\ge4}(2C)^{-r/2} = \frac1{4C-2(2C)^{1/2}} \le 1/4 < 1/2.\]
		This proves \cref{lem:sparse}.
	\end{proof}
	
	Combining \cref{lem:absorb,lem:sparse} then gives the following.
	
	\begin{lemma}[{cf.~\cite[Lemma~3.3]{KK26a}}]
		\label{lem:bounded}
		Let $\Delta\ge1$.
		If $n_{i-1}^3\ge(288\Delta C)\cdot n_1^2n_i$ for all $i\in\{2,\ldots,C\}$, then with probability greater than $1/2$, all subgraphs of $G$ with maximum degree at most $\Delta$ have average degree less than $6$.
	\end{lemma}
	
	\begin{proof}
		Let $E$ be the event that for all $i\in[C]$ and all $X\subset L_i$ with $\abs X\le4\Delta n_i$, $G[X]$ has at most $3\abs X/2$ edges.
		Since $(n_i)_{i=1}^C$ satisfies \eqref{eq:dexp} with $K=4\Delta$, \cref{lem:sparse} implies that $E$ occurs with probability greater than $1/2$.
		It suffices to show that if $E$ occurs, then every subgraph $Q$ of $G$ with $\abs Q\ge1$ and maximum degree at most $\Delta$ has average degree less than $6$.
		Hence, by \cref{lem:absorb}, it remains to prove that if $E$ occurs then every such $Q$ admits a vertex partition $(X,Y,Z)$ such that:
		\begin{itemize}
			\item $Q[X]$ has at most $3\abs X/2$ edges;
			\item $Y$ is stable in $Q$;
			\item $Q$ has at most $\abs X$ edges between $X$ and $Y$; and
			\item $\abs Z<\abs Q/(2\Delta)$.
		\end{itemize}
		
		To this end, let $R_0:=V(G)$, and let $i\in\{0,1,\ldots,C\}$ be minimal such that $\abs{R_i}<\abs Q/(2\Delta)$; this clearly holds for $i=C$ as $R_C=\emptyset$.
		Then $i\ge1$ since $\abs Q\le\abs G=\abs{R_0}$, and the minimality of $i$ yields
		\[\abs Q \le 2\Delta\abs{R_{i-1}} = 2\Delta(n_i+n_{i+1}+\cdots+n_C) \le 4\Delta n_i.\]
		
		Now, as discussed, let $X:=V(Q)\cap L_i$, $Y:=V(Q)\cap V_i$, and $Z:=V(Q)\cap R_i$; then $(X,Y,Z)$ is a partition of $V(Q)$ and $\abs Z<\abs Q/(2\Delta)$.
		By the construction of $G$, every vertex in $X$ sends at most one edge to $Y$; and so $Q$ has at most $\abs X$ edges between $X$ and $Y$.
		Also, $Y\subset V_i$ is stable.
		Now, since $\abs X\le\abs Q\le4\Delta n_i$ and $E$ occurs, $Q[X]$ has at most $3\abs X/2$ edges.
		This proves \cref{lem:bounded}.
	\end{proof}
	
	We are now ready to prove \cref{thm:orient1} by suitably choosing $n_1,\ldots,n_C$ to satisfy the double-exponential decay condition \eqref{eq:dexp}.
	This choice, up to specific constants, is common to all previous works using the discussed multipartite random graph model or its variations~\cite{CJMM26,DOW20,JSS26,PRS95,Ste25}, and is also used in \cite[Section~3]{KK26a} and \cite[Subsection~5.3]{KK26b}.
	
	\begin{proof}[Proof of \cref{thm:orient1}]
		Let $C:=16q^2$, let $D:=\ceil{(288\Delta C)^{1/2}}$, and let $n_i:=D^{3^{C-1}-3^{i-1}}$ for every $i\in[C]$.
		Then, for all $i\in\{2,\ldots,C\}$, we have
		\[\begin{aligned}
			n_{i-1} &= D^{3^{C-1}-3^{i-2}} = D^{2\cdot3^{i-2}}n_i\ge D^2n_i \ge 2n_i,\quad\text{and}\\
			n_{i-1}^3 &= D^{3^C-3^{i-1}} = D^2\cdot D^{2\cdot 3^{C-1}-2}\cdot D^{3^{C-1}-3^{i-1}} = D^2n_1^2n_i \ge (288\Delta C)\cdot n_1^2n_i.
		\end{aligned}\]
		Thus, by \cref{lem:stable,lem:bounded}, with positive probability, the random graph $G$ satisfies $\chis(G)>q$, and all subgraphs of $G$ with maximum degree at most $\Delta$ have average degree less than $6$.
		By construction, since $C=16q^2$, the graph $G$ is $(16q^2-1)$-degenerate.
		This proves \cref{thm:orient1}.
	\end{proof}
	
	\section{\texorpdfstring{$2$-Degeneracy}{2-Degeneracy}}
	\label{sec:2degen}
	
	This section provides a proof of \cref{thm:orient}, by detailing necessary changes to \cref{sec:random} and to the construction in \cite[Section~3]{KK26a} and \cite[Sections 4 and~5]{KK26b}.
	In what follows, for a graph $G$ with vertex set $V(G)$ and edge set $E(G)$, let $\edge(G):=\abs{E(G)}$.
	For every $S\subset V(G)$, let $\edge_G(S):=\edge(G[S])$; and for all disjoint $A,B\subset V(G)$ let $\edge_G(A,B)$ be the number of edges of $G$ between $A$ and $B$.
	For every $v\in V(G)$, let $d_G(v)$ be the degree of $v$ in $G$.
	
	\subsection{Adding a stability condition}
	
	Given the discussion in \cref{sec:random}, to prove \cref{thm:orient} it would be natural to modify \cref{lem:absorb} so that its conclusion becomes `$Q$ has minimum degree at most $2$'.
	As a first step, it is not hard to see that the assumption \eqref{eq:dexp} on the sequence $(n_i)_{i=1}^C$ can be adjusted so that the density $3/2$ in \cref{lem:sparse} (and so in the first property of \cref{lem:absorb}) is improved to $1+\eps$ for any prescribed $\eps>0$ (see \cref{lem:verysparse} below).
	A suitable change in the parameter $K$ in \eqref{eq:dexp} also allows one to relax the last property of \cref{lem:absorb} to `$\abs Z<\eps\abs Q$'.
	These two observations were actually optimised in \cite{KK26a} and~\cite{KK26b} to prove \cref{thm:orient1} with average degree less than $4$ (so minimum degree at most $3$), and in general cannot yield minimum degree at most $2$.
	To explain the latter, $Q$ could have minimum degree $3$ and average degree less than $4$ in the following scenario:
	\begin{itemize}
		\item $Q[X]$ is a cycle;
		\item every vertex in $X$ has exactly one neighbour in $Y$;
		\item $Y$ is negligible compared to $X$, and $Z$ is empty; and
		\item all vertices in $Y$ have degree more than $3$ in $Q$.
	\end{itemize}
	
	Therefore, to ensure that $Q$ has minimum degree at most $2$, it is necessary to impose another structural hypothesis on \cref{lem:absorb}.
	The following lemma incorporates one such hypothesis, by requiring the vertices in $X$ with a neighbour in $Y$ to be pairwise nonadjacent.
	
	\begin{lemma}
		\label{lem:tri}
		Let $\Delta\ge1$.
		Let $Q$ be a graph with $\abs Q\ge1$ and maximum degree at most $\Delta$.
		Assume that $V(Q)$ admits a partition $(X,Y,Z)$ allowing empty parts, such that:
		\begin{itemize}
			\item $\edge_Q(X)\le11\abs X/10$;
			\item $Y$ is stable;
			\item every vertex in $X$ has at most one neighbour in $Y$;
			\item the set of vertices in $X$ with a neighbour in $Y$ is stable; and
			\item $\abs Z<\abs Q/(10\Delta)$.
		\end{itemize}
		Then $Q$ has minimum degree at most $2$.
	\end{lemma}
	
	\begin{proof}
		Suppose that $Q$ has minimum degree at least $3$.
		Our aim is to reach a contradiction by showing that each of $\abs X,\abs Y,\abs Z$ is small compared to $\abs Q$.
		We already know $\abs Z<\abs Q/(10\Delta)$ by the assumption.
		For $\abs Y$, since $Y$ is stable and $Q$ has maximum degree at most $\Delta$, taking the sum of degrees of vertices in $Y$ gives
		\begin{equation}
			\label{eq:y}
			3\abs Y \le \sum_{v\in Y}d_Q(v) \le \edge_Q(X,Y)+\edge_Q(Y,Z) \le \abs X+\Delta\abs Z,
		\end{equation}
		since every vertex in $X$ sends at most one edge to $Y$.
		It remains to bound $\abs X$ via the following claim.
		
		\begin{claim}
			$\abs X\le6\Delta\abs Z$.
		\end{claim}
		
		\begin{subproof}
			The key quantity to look at is $b:=\edge_Q(X)\le11\abs X/10$.
			Let $W$ be the set of vertices in $X$ with a neighbour in $Y$; then $W$ is stable.
			Thus, taking the sum of degrees of vertices in $W$ gives
			\[3\abs W \le \sum_{v\in W}d_Q(v) = \edge_Q(W,X\setminus W)+\edge_Q(W,Y)+\edge_Q(W,Z) \le b+\abs W+\Delta\abs Z,\]
			which implies $2\abs W\le b+\Delta\abs Z$.
			Now, the definition of $W$ gives $\edge_Q(X,Y)=\abs W$.
			Taking the sum of degrees of vertices in $X$ then gives
			\[\begin{aligned}
				3\abs X \le \smash[b]{\sum_{v\in X}d_Q(v)} &= 2\edge_Q(X)+\edge_Q(X,Y)+\edge_Q(X,Z)\\
				&\le 2b+\abs W+\Delta\abs Z
				\le 5b/2+3\Delta\abs Z/2
				\le 11\abs X/4+3\Delta\abs Z/2,
			\end{aligned}\]
			and so $\abs X\le6\Delta\abs Z$.
			This proves the claim.
		\end{subproof}
		
		Now, \eqref{eq:y} and the claim above together imply that
		\[3\abs Q = 3\abs X+3\abs Y+3\abs Z \le 4\abs X+(\Delta+3)\abs Z \le (25\Delta+3)\abs Z \le 28\Delta\abs Z,\]
		which violates the assumption $\abs Z<\abs Q/(10\Delta)$.
		This proves \cref{lem:tri}.
	\end{proof}
	
	\subsection{Modifying the random multipartite construction}
	\label{subsec:random}
	
	We now modify the random multipartite construction in \cref{sec:random} so that the fourth hypothesis of \cref{lem:tri} holds, as follows.
	Let $C\ge2$ be an integer, and let $(n_i)_{i=1}^C$ be a decreasing sequence satisfying \eqref{eq:exp}.
	Let $\Gamma$ be a graph with vertex set $[C]$ (to be chosen later), and let $G$ be a random multipartite graph with parts $V_1,\ldots,V_C$, such that:
	\begin{itemize}
		\item $\abs{V_i}=n_i$ for all $i\in[C]$; and
		\item independently for each adjacent pair $i,j$ in $\Gamma$ with $i<j$, and for each $u\in V_i$, choose a uniformly random vertex $v\in V_j$, then make $uv$ an edge of $G$.
	\end{itemize}
	
	Then the random graph in \cref{sec:random} is the special case of this construction when $\Gamma$ is the complete $C$-vertex graph.
	In the case of this subsection, making the `template' graph $\Gamma$ triangle-free would be a natural way to ensure the fourth hypothesis of \cref{lem:tri}.
	Thus, we will choose $\Gamma$ from the following lemma, whose proof is a routine deletion argument and can be found in \cref{sec:triangle}.
	
	\begin{lemma}
		\label{lem:triangle}
		For every $r\ge4$, there exists a triangle-free graph $\Gamma$ with maximum degree at most $144r\ln r$ and with no stable set of size at least $\abs\Gamma/r$.
	\end{lemma}
	
	From now on, we fix $r=(4q)^2$, let $\Gamma$ be given by \cref{lem:triangle}, let $C:=\abs{\Gamma}$, and identify $V(\Gamma)$ with $[C]$.
	Our first task is to verify the following analogue of \cref{lem:stable}, by an almost identical argument in which the definition of good subsets is refined slightly.
	
	\begin{lemma}
		\label{lem:stable1}
		Assume \eqref{eq:exp}.
		Then $\chis(G)>q$ with probability greater than $1/2$.
	\end{lemma}
	
	\begin{proof}
		For each $i\in[C]$, let $N_\Gamma[i]:=N_\Gamma(i)\cup\{i\}$ where $N_\Gamma(i)$ is the neighbourhood of $i$ in $\Gamma$; and for each $v\in V_i$, let $w(v):=1/n_i$.
		For each $S\subset V(G)$, let $w(S):=\sum_{v\in S}w(v)$.
		Thus $w(V_i)=1$ for all $i\in[C]$, and $w(V(G))=C$.
		Let $w_{\max}$ be the maximum $w(S)$ over all stable sets $S\subset V(G)$.
		If $w_{\max}<C/q$, then the weight function $w/w_{\max}$ witnesses $\chis(G)\ge C/w_{\max}>q$, by the definition of fractional chromatic number.
		Therefore, it suffices to show that $G$ has a stable set of weight at least $C/q$ with probability less than $1/2$.
		
		To see this, we say that a subset $S\subset V(G)$ is {\em good} if the following hold:
		\begin{itemize}
			\item $w(S)\ge8q$; and
			\item there exists $I\subset[C]$ with $w(S\cap V_j)=0$ for all $j\in[C]\setminus I$ and $w(S\cap V_j)\ge(2q)^{-1}$ for all $j\in I$,
			such that $I\subset N_\Gamma[i]$ where $i$ is the least element of $I$.
		\end{itemize}
		
		\begin{claim}
			Every $S\subset V(G)$ with $w(S)\ge C/q$ contains a good subset.
		\end{claim}
		
		\begin{subproof}
			Let $I:=\{j\in[C]:w(S\cap V_j)\ge(2q)^{-1}\}$; then $\sum_{j\in I}w(S\cap V_j)\ge w(S)-C\cdot(2q)^{-1}\ge C/(2q)$.
			Let $R:=\bigcup_{j\in I}(S\cap V_j)$; then $w(R)\ge C/(2q)$.
			By a greedy argument from left to right along $I$, we obtain a stable set $J$ in $\Gamma[I]$ such that every vertex in $I\setminus J$ has a left-hand neighbour in $J$.
			For each $i\in J$, let $P_i$ consist of $i$ and its right-hand neighbours in $\Gamma[I]$; then $I=\bigcup_{i\in J}P_i$.
			Thus, since $\abs J<\abs\Gamma/r=C(4q)^{-2}$, there exists $i\in J$ with $\sum_{j\in P_i}w(S\cap V_j)\ge w(R)/\abs J\ge8q$.
			Then $S'=\bigcup_{j\in P_i}(S\cap V_j)$ is a good subset of $S$.
			This proves the claim.
		\end{subproof}
		
		By the claim above, to prove the lemma, it remains to show that $G$ has a good stable set with probability less than $1/2$.
		To this end, for each $i\in[C]$ let $B_i$ be the set of good $S\subset V(G)$ for which $i$ is the least index satisfying $w(S\cap V_i)\ge1/(2q)$; then $S\subset V_i\cup V_{i+1}\cup\cdots\cup V_C$ for all $S\in B_i$.
		Now, fix $i\in[C]$, and note that $\abs{B_i}\le2^{n_i+n_{i+1}+\cdots+n_C}\le2^{2n_i}$.
		Let $S\in B_i$ with $I=\{j\in[C]:w(S\cap V_j)\ge1/(2q)\}$; then $i$ is the least element in $I$.
		Let $w_j:=w(S\cap V_j)=\abs{S\cap V_j}/n_j\le1$ for all $j\in I$; then $\abs I\ge w(S)\ge8q>1$.
		If $S$ is stable in $G$, then $G$ contains no edge between $S\cap V_i$ and $\bigcup_{j\in I\setminus\{i\}}(S\cap V_j)$; and the latter event occurs with the following probability (note that $w_i\cdot w(S)-w_i^2\ge(2q)^{-1}\cdot 8q-1=3$):
		\[\prod_{j\in I\setminus\{i\}}(1-w_j)^{\abs{S\cap V_i}}
		\le e^{-\abs{S\cap V_i}\sum_{j\in I\setminus\{i\}}w_j}
		= e^{-w_in_i\cdot(w(S)-w_i)}
		= e^{-(w_i\cdot w(S)-w_i^2)n_i}
		\le e^{-3n_i}.\]
		
		Since the above holds for all $i\in[C]$, the union bound implies that $G$ contains a good stable set with probability at most
		\[\sum_{i\in[C]}\abs{B_i}\cdot e^{-3n_i}
		\le \sum_{i\in[C]}2^{2n_i}\cdot e^{-3n_i}
		\le \sum_{i\in[C]}2^{-2n_i}
		< \sum_{n\ge1}2^{-2n} < 1/2,\]
		where the third inequality holds since we are assuming \eqref{eq:exp}.
		This proves \cref{lem:stable1}.
	\end{proof}
	
	To arrange for the first hypothesis of \cref{lem:tri},
	we will formulate a more general version of \cref{lem:sparse}.
	The constant factor $3/2$ there will be replaced by $1+\eps$, where $\eps\in(0,\frac12]$ is some parameter to be chosen later.
	In this setting, for $K\ge1$, we assume the following faster double-exponential decay than \eqref{eq:dexp}:
	\begin{equation}
		\label{eq:sdexp}
		n_{i-1}^{1+\eps} \ge (3^{2+\eps}K^\eps C)\cdot n_1n_i^\eps \quad\text{for all }i\in\{2,\ldots,C\}.
	\end{equation}
	
	We also recall the definitions
	\[L_i := \bigcup_{1\le j<i}V_j \quad\text{and}\quad
	R_i := \bigcup_{i<j\le C}V_j \quad\text{for all }i\in[C].\]
	
	Assuming \eqref{eq:sdexp}, the following lemma is an extension of \cref{lem:sparse} with the same proof.
	
	\begin{lemma}
		\label{lem:verysparse}
		Assume \eqref{eq:sdexp}.
		Then $G$ satisfies the following with probability greater than $1/2$: for all $i\in[C]$ and all $X\subset L_i$ with $\abs X\le Kn_i$, $G[X]$ has at most $(1+\eps)\abs X$ edges.
	\end{lemma}
	
	\begin{proof}
		For each $i\in\{2,\ldots,C\}$ and each integer $r$ with $4\le r\le Kn_i$ (in particular $r\le Kn_i\le2n_{i-1}/3$), let $E_{i,r}$ be the event that there exists $X\subset L_i$ such that $\abs X=r$ and $G[X]$ has at least $(1+\eps)r$ edges.
		Every pair of vertices in $L_i$ becomes an edge of $G$ with probability at most $1/n_{i-1}$.
		For every set $P$ of pairs of vertices in $L_i$, the events that the pairs in $P$ become edges of $G$ are independent unless $P$ contains two pairs $uv$ and $uv'$ with $u\in V_j$ and $v,v'\in V_s$ where $1\le j<s<i$; such two pairs cannot be edges of $G$ simultaneously.
		Consequently, the probability that all pairs in $P$ become edges of $G$ is at most $n_{i-1}^{-\abs P}$.
		Since $\abs{L_i}=n_1+\cdots+n_{i-1}\le2n_1$ and by the inequality $\binom nk\le(\frac{3n}k)^k$ for all $n\ge k\ge1$, the union bound over all sets $X\subset L_i$ with $\abs X=r$, and all sets $P$ of $\ceil{(1+\eps)r}$ pairs of vertices in $X$, implies that $E_{i,r}$ occurs with probability at most
		\[\begin{aligned}
			\binom{\abs{L_i}}r\cdot\binom{r(r-1)/2}{\ceil{(1+\eps)r}}\cdot n_{i-1}^{-\ceil{(1+\eps)r}}
			&\le \left(\frac{6n_1}r\right)^r\cdot \left(\frac{3r}2\right)^{\ceil{(1+\eps)r}}\cdot n_{i-1}^{-\ceil{(1+\eps)r}}\\
			&\le \left(\frac{6n_1}r\right)^r\cdot \left(\frac{3r}{2n_{i-1}}\right)^{(1+\eps)r}
			= \left(\frac{3^{2+\eps}n_1r^\eps}{2^\eps n_{i-1}^{1+\eps}}\right)^r
			\le\left(\frac{3^{2+\eps}K^\eps n_1n_i^\eps}{n_{i-1}^{1+\eps}}\right)^r
			\le C^{-r}
		\end{aligned}\]
		where the last inequality holds since we are assuming \eqref{eq:sdexp}.
		Thus, since $L_1=\emptyset$, summing over $i\in[C]$ and $4\le r\le Kn_i$ and applying the union bound imply that the probability that there exist $i\in[C]$ and $X\subset L_i$ such that $\abs X\le Kn_i$ and $G[X]$ has more than $(1+\eps)\abs X$ edges is at most
		\[C\cdot\sum_{r\ge4}C^{-r} = \frac1{C^3-C^2} \le 1/4 < 1/2.\]
		This proves \cref{lem:verysparse}.
	\end{proof}
	
	\subsection{Putting everything together}

	We are now ready to prove \cref{thm:orient}, which we restate here for the reader's convenience.
	
	\thmorient*
	
	\begin{proof}
		Let $r:=(4q)^2$, let $\Gamma$ be the triangle-free graph given by \cref{lem:triangle}, and let $C:=\abs\Gamma$.
		Let $\eps:=1/10$, $K:=20\Delta$, and $D:=\ceil{3^{2+\eps}K^\eps C}$.
		We next choose suitable $n_1,\ldots,n_C$ so that \eqref{eq:sdexp} holds for $\eps$ and $K$, as follows:
		\[\begin{aligned}
			n_1 &:= D^{(1+\eps^{-1})^{C-1}}, &&\text{and}\\
			n_i &:= D^{1-(1+\eps^{-1})^{i-1}}n_1 &&\text{for all }i\in[C].
		\end{aligned}\]
		Then, for all $i\in\{2,\ldots,C\}$, we have
		\[\begin{aligned}
			(3^{2+\eps}K^\eps C)\cdot n_1n_i^\eps
			\le Dn_1n_i^\eps
			&= Dn_1\cdot D^{\eps-\eps(1+\eps^{-1})^{i-1}}n_1^\eps\\
			&= D^{1+\eps-\eps(1+\eps^{-1})^{i-1}}n_1^{1+\eps}
			= \left(D^{1-(1+\eps^{-1})^{i-2}}n_1\right)^{1+\eps}
			= n_{i-1}^{1+\eps},
		\end{aligned}\]
		which gives $n_{i-1}^\eps\ge K^\eps n_i^\eps\ge2^\eps n_i^\eps$ and so $n_{i-1}\ge2n_i$.
		Hence \eqref{eq:exp} and \eqref{eq:sdexp} are satisfied.
		
		Now, let $G$ be the random multipartite graph constructed in \cref{subsec:random}, given $\Gamma$ and $(n_i)_{i=1}^C$.
		By \cref{lem:stable1,lem:verysparse}, the random graph $G$ satisfies the following with positive probability:
		\begin{itemize}
			\item $\chis(G)>q$; and
			\item for all $i\in[C]$ and all $X\subset L_i$ with $\abs X\le Kn_i$, $G[X]$ has at most $(1+\eps)\abs X$ edges.
		\end{itemize}
		
		Choose such an instance of $G$.
		The triangle-free graph $\Gamma$ given by \cref{lem:triangle} has maximum degree at most $144r\ln r=144\cdot(4q)^2\cdot2\ln(4q)=4608q^2\ln(4q)$.
		Order all vertices of $V_i$ before all vertices of $V_j$ whenever $i<j$, in an arbitrary order within each part; and this ordering witnesses that $G$ is $\floor{4608q^2\ln(4q)}$-degenerate.
		This and the first property of $G$ above prove the first two conditions in \cref{thm:orient}.
		To prove the last condition, namely, that every subgraph of $G$ with maximum degree at most $\Delta$ is $2$-degenerate, it suffices to show that every such subgraph $Q$ with $\abs Q\ge1$ has minimum degree at most $2$, because $2$-degeneracy then follows by straightforward induction.
		
		Let $Q$ be a subgraph of $G$ with $\abs Q\ge1$ and maximum degree at most $\Delta$.
		By \cref{lem:tri}, it suffices to show that $V(Q)$ admits a partition $(X,Y,Z)$ allowing empty parts, such that:
		\begin{itemize}
			\item $\edge_Q(X)\le11\abs X/10$;
			\item $Y$ is stable;
			\item every vertex in $X$ has at most one neighbour in $Y$;
			\item the set of vertices in $X$ with exactly one neighbour in $Y$ is stable; and
			\item $\abs Z<\abs Q/(10\Delta)$.
		\end{itemize}
		
		To this end, let $R_0:=V(G)$, and let $i\in\{0,1,\ldots,C\}$ be minimal such that $\abs{R_i}<\abs Q/(10\Delta)$; this clearly holds for $i=C$ as $R_C=\emptyset$.
		Then $i\ge1$ since $\abs Q\le\abs G=\abs{R_0}$.
		Let $X:=V(Q)\cap L_i$, $Y:=V(Q)\cap V_i$, and $Z:=V(Q)\cap R_i$.
		Then $Y$ is stable, $\abs Z<\abs Q/(10\Delta)$, and every vertex in $X$ has at most one neighbour in $Y$ by the construction of $G$.
		Also, since $\Gamma$ is triangle-free and $Y\subset V_i$, the set of vertices in $X$ with exactly one neighbour in $Y$ is stable.
		Now, the minimality of $i$ and the choice of $K$ imply that
		\[\abs X\le\abs Q\le10\Delta\abs{R_{i-1}}\le10\Delta(n_i+n_{i+1}+\cdots+n_C)\le20\Delta n_i=Kn_i.\]
		Hence $\edge_Q(X)\le\edge_G(X)\le(1+\eps)\abs X=11\abs X/10$ by the choice of $\eps$ and the second property of $G$ above.
		This proves \cref{thm:orient}.
	\end{proof}
	
	\appendix
	\crefalias{section}{appendixsection}
	
	\section{Proof of \cref{lem:triangle}}
	\label{sec:triangle}
	
	We give a proof of \cref{lem:triangle} using the probabilistic method and a standard deletion argument.
	
	\begin{proof}[Proof of \cref{lem:triangle}]
		Let $d:=24r\ln r$, and let $N\ge6d^3$ be an integer.
		Consider the Erd\H os--R\'enyi random graph $G$ on $N$ vertices with edge probability $p=d/N$.
		The expected numbers of edges and triangles in $G$ are, respectively, $\binom N2p\le Nd/2$ and $\binom N3p^3\le d^3/6$.
		Thus, by Markov's inequality, $G$ contains at least $Nd$ edges or at least $d^3$ triangles with probability at most $1/2+1/6=2/3$.
		Now, for $s:=\ceil{N/(2r)}\ge N/(2r)\ge16$, we have $\binom s2\ge s^2/4$, and $ps/4-\ln(6r)\ge3\ln r-\ln(6r)=2\ln r-\ln 6\ge\ln(8/3)>1/2$.
		Thus, the probability that $G$ has a stable set of size $s$ is at most
		\[\binom Ns(1-p)^{\binom s2} \le \left(\frac{3N}s\right)^se^{-ps^2/4}\le (6r)^se^{-ps^2/4} = e^{-s(ps/4-\ln(6r))} < 1/3.\]
		Hence, with positive probability, $G$ has fewer than $Nd$ edges, fewer than $d^3$ triangles, and no stable set of size $s$.
		Choose such an instance of $G$.
		Then $G$ has fewer than $N/3$ vertices of degree at least $6d$, and contains a set of at most $d^3\le N/6$ vertices hitting all triangles in $G$.
		Removing all such vertices gives a triangle-free subgraph $\Gamma$ of $G$ with $\abs\Gamma\ge\abs G/2$, maximum degree at most $6d$, and no stable set of size $s$.
		Since $s-1<N/(2r)\le\abs\Gamma/r$, this proves \cref{lem:triangle}.
	\end{proof}
	
	\section*{Acknowledgement}
	
	We thank Raphael Steiner for sharing \cref{thm:chi-avoidable} with us and letting us include it in this note.
	
	\section*{Declaration of AI use}
	
	The initial purpose of this note was purely an exposition of the proof of \cref{thm:girthchi}, which the authors had learned of through its announcement by Conjectures.io~\cite{KK26b}.
	While trying to digest the construction method and link it to relevant existing ideas in the literature, the authors realised that \cite{KK26b} implicitly proves a variant of \cref{thm:orient1} with `$6$' replaced by `$4$'.
	The authors also observed that \cref{thm:orient1} itself suffices to improve \cref{thm:girthchi} from $6$-colourability to $4$-colourability via \cref{lem:extend}, just like a suitable analogue of \cref{thm:orient1} for $2$-degeneracy (or $3$-colourability) such as \cref{thm:orient} would suffice to prove \cref{thm:main}.
	That work was done entirely by the authors based on the manuscript~\cite{KK26b}, with no AI tools involved.
	After that, the authors prompted ChatGPT-6 Astra to prove a $2$-degeneracy (or $3$-colourability) version of \cref{thm:orient1} by modifying the probabilistic construction in~\cite{KK26b}.
	ChatGPT-6 Astra discovered a proof of \cref{thm:orient}, in particular proposing \cref{lem:tri} as a suitable extension of \cref{lem:absorb}.
	The authors streamlined and presented the proof in \cref{sec:2degen}.
	Then the authors discovered the manuscript~\cite{KK26a} and revised the text accordingly.
	ChatGPT-5.6 Sol was also used to proofread this note.
	
	\bibliographystyle{abbrv}
	\bibliography{girthchi}

\begin{thebibliography}{10}

\bibitem{ABHSZ26}
S.~Adenwalla, S.~Braunfeld, T.~Hons, J.~Sylvester, and V.~Zamaraev.
\newblock Set-defined graph classes: {$\chi$}-boundedness meets tropical
  algebra.
\newblock \href{https://arxiv.org/abs/2607.23754v1}{\tt arXiv:2607.23754v1},
  2026.

\bibitem{Apa22}
R.~Aparecido~Enju.
\newblock Uma conjectura de {E}rd{\H{o}}s e {H}ajnal.
\newblock Master's thesis, Instituto de Matem{\'{a}}tica, Estat{\'{\i}}stica e
  Ci{\^{e}}ncia da Computa{\c{c}}{\~{a}}o, Universidade de S{\~{a}}o Paulo,
  2022.
\newblock In Portuguese.

\bibitem{CJMM26}
D.~Chakraborti, O.~Janzer, A.~Methuku, and R.~Montgomery.
\newblock Regular subgraphs at every density.
\newblock {\em Trans. Amer. Math. Soc.}, 379(11):8069--8090, 2026.

\bibitem{CGL06}
M.~Cropper, A.~Gy{\'{a}}rf{\'{a}}s, and J.~Lehel.
\newblock Hall ratio of the {M}ycielski graphs.
\newblock {\em Discrete Math.}, 306(16):1988--1990, 2006.

\bibitem{DOW20}
Z.~Dvo{\v{r}}{\'{a}}k, P.~Ossona~de Mendez, and H.~Wu.
\newblock {$1$}-subdivisions, the fractional chromatic number and the {H}all
  ratio.
\newblock {\em Combinatorica}, 40(6):759--774, 2020.

\bibitem{Erd59}
P.~Erd{\H{o}}s.
\newblock Graph theory and probability.
\newblock {\em Canadian J. Math.}, 11:34--38, 1959.

\bibitem{Erd69}
P.~Erd{\H{o}}s.
\newblock Problems and results in chromatic graph theory.
\newblock In F.~Harary, editor, {\em Proof Techniques in Graph Theory}, pages
  27--35. Academic Press, New York, 1969.

\bibitem{Erd71}
P.~Erd{\H{o}}s.
\newblock Some unsolved problems in graph theory and combinatorial analysis.
\newblock In D.~J.~A. Welsh, editor, {\em Combinatorial Mathematics and its
  Applications}, pages 97--109. Academic Press, London, 1971.

\bibitem{Erd75}
P.~Erd{\H{o}}s.
\newblock Problems and results on finite and infinite graphs.
\newblock In A.~Hajnal, R.~Rado, and V.~T.~S{\'{o}}s, editors, {\em Infinite
  and finite sets}, volume~10 of {\em Colloquia Mathematica Societatis
  J{\'{a}}nos Bolyai}, pages 403--424. North-Holland, Amsterdam, 1975.

\bibitem{Erd76a}
P.~Erd{\H{o}}s.
\newblock Problems and results in combinatorial analysis.
\newblock In {\em Colloquio Internazionale sulle Teorie Combinatorie (Roma,
  1973), Tomo II}, volume~17 of {\em Atti dei Convegni Lincei}, pages 1--10.
  Accademia Nazionale dei Lincei, Rome, 1976.

\bibitem{Erd76b}
P.~Erd{\H{o}}s.
\newblock Problems and results in graph theory and combinatorial analysis.
\newblock In C.~S. J.~A. Nash-Williams and J.~Sheehan, editors, {\em
  Proceedings of the 5th British Combinatorial Conference}, volume~XV of {\em
  Congressus Numerantium}, pages 169--192. Utilitas Mathematica, Winnipeg,
  1976.

\bibitem{Erd78}
P.~Erd{\H{o}}s.
\newblock Problems and results in graph theory and combinatorial analysis.
\newblock In {\em Probl{\`{e}}mes combinatoires et th{\'{e}}orie des graphes},
  volume 260 of {\em Colloques Internationaux du CNRS}, pages 127--129.
  {\'{E}}ditions du CNRS, Paris, 1978.

\bibitem{Erd79a}
P.~Erd{\H{o}}s.
\newblock Problems and results in graph theory and combinatorial analysis.
\newblock In J.~A. Bondy and U.~S.~R. Murty, editors, {\em Graph theory and
  related topics}, pages 153--163. Academic Press, New York, 1979.

\bibitem{Erd79b}
P.~Erd{\H{o}}s.
\newblock Some old and new problems in various branches of combinatorics.
\newblock In F.~Hoffman, D.~McCarthy, R.~C. Mullin, and R.~G. Stanton, editors,
  {\em Proceedings of the 10th Southeastern International Conference on
  Combinatorics, Graph Theory, and Computing}, volume XXIII of {\em Congressus
  Numerantium}, pages 19--37. Utilitas Mathematica, Winnipeg, 1979.

\bibitem{Erd81a}
P.~Erd{\H{o}}s.
\newblock On the combinatorial problems which {I} would most like to see
  solved.
\newblock {\em Combinatorica}, 1(1):25--42, 1981.

\bibitem{Erd81b}
P.~Erd{\H{o}}s.
\newblock Problems and results on finite and infinite combinatorial analysis
  {II}.
\newblock {\em Enseign. Math. (2)}, 27(1--2):163--176, 1981.

\bibitem{Erd82}
P.~Erd{\H{o}}s.
\newblock Problems and results on finite and infinite combinatorial analysis
  {II}.
\newblock In {\em Logic and Algorithmic}, volume~30 of {\em Monographies de
  l'Enseignement Math{\'{e}}matique}, pages 131--144. Universit{\'{e}} de
  Gen{\`{e}}ve, Geneva, 1982.

\bibitem{Erd84}
P.~Erd{\H{o}}s.
\newblock Some new and old problems on chromatic graphs.
\newblock In K.~S. Vijayan and N.~M. Singhi, editors, {\em Combinatorics and
  Applications}, pages 118--126. Indian Statistical Institute, Calcutta, 1984.

\bibitem{Erd85}
P.~Erd{\H{o}}s.
\newblock Problems and results on chromatic numbers in finite and infinite
  graphs.
\newblock In Y.~Alavi, G.~Chartrand, D.~R. Lick, C.~E. Wall, and L.~Lesniak,
  editors, {\em Graph theory with applications to algorithms and computer
  science}, pages 201--213. John Wiley \& Sons, New York, 1985.

\bibitem{Erd90}
P.~Erd{\H{o}}s.
\newblock Some of my favourite unsolved problems.
\newblock In A.~Baker, B.~Bollob{\'{a}}s, and A.~Hajnal, editors, {\em A
  Tribute to Paul Erd{\H{o}}s}, pages 467--478. Cambridge University Press,
  Cambridge, 1990.

\bibitem{Erd95}
P.~Erd{\H{o}}s.
\newblock On some problems in combinatorial set theory.
\newblock {\em Publ. Inst. Math. (Beograd) (N.S.)}, 57(71):61--65, 1995.

\bibitem{EH66}
P.~Erd{\H{o}}s and A.~Hajnal.
\newblock On chromatic number of graphs and set-systems.
\newblock {\em Acta Math. Acad. Sci. Hungar.}, 17(1--2):61--99, 1966.

\bibitem{EH85}
P.~Erd{\H{o}}s and A.~Hajnal.
\newblock Chromatic number of finite and infinite graphs and hypergraphs.
\newblock {\em Discrete Math.}, 53:281--285, 1985.

\bibitem{HN60}
F.~Harary and R.~Z. Norman.
\newblock Some properties of line digraphs.
\newblock {\em Rend. Circ. Mat. Palermo (2)}, 9(2):161--168, 1960.

\bibitem{HE72}
C.~C. Harner and R.~C. Entringer.
\newblock Arc colorings of digraphs.
\newblock {\em J. Combin. Theory Ser. B}, 13(3):219--225, 1972.

\bibitem{JSS26}
B.~Janzer, R.~Steiner, and B.~Sudakov.
\newblock Chromatic number and regular subgraphs.
\newblock {\em Bull. Lond. Math. Soc.}, 58(4):Article e70262, 2026.

\bibitem{KK26a}
J.~Kohlmeyer and L.~Kruer.
\newblock A counterexample to {E}rd{\H{o}}s problem 108 via arc graphs.
\newblock \url{https://jenwin.io/papers/erdos108-arc-graphs.pdf}, 2026.
\newblock Retrieved 27th September 2026.

\bibitem{KK26b}
L.~Kruer and J.~Kohlmeyer.
\newblock Erd{\H{o}}s problem 108: {H}igh chromatic number with six-colour
  {$C_4$}-free subgraphs.
\newblock \url{https://conjectures.io/papers/erdos108.pdf}, 2026.
\newblock Retrieved 17th September 2026.

\bibitem{MW23}
B.~Mohar and H.~Wu.
\newblock Subgraphs of {K}neser graphs with large girth and large chromatic
  number.
\newblock {\em Art Discrete Appl. Math.}, 6(2):Article 2.11, 2023.

\bibitem{PTW26}
S.~Pettie, G.~Tardos, and B.~Walczak.
\newblock On a clique game and the {E}rd{\H{o}}s--{H}ajnal problem on
  high-chromatic high-girth subgraphs.
\newblock In K.~G. Larsen and B.~Saha, editors, {\em Proceedings of the 2026
  {A}nnual {ACM}-{SIAM} {S}ymposium on {D}iscrete {A}lgorithms ({SODA})}, pages
  2903--2927. SIAM, Philadelphia, 2026.

\bibitem{PRS95}
L.~Pyber, V.~R{\"{o}}dl, and E.~Szemer{\'{e}}di.
\newblock Dense graphs without {$3$}-regular subgraphs.
\newblock {\em J. Combin. Theory Ser. B}, 63(1):41--54, 1995.

\bibitem{Rod77}
V.~R{\"{o}}dl.
\newblock On the chromatic number of subgraphs of a given graph.
\newblock {\em Proc. Amer. Math. Soc.}, 64(2):370--371, 1977.

\bibitem{Sad25}
A.~Sadhukhan.
\newblock Shift graphs, chromatic number and acyclic one-path orientations.
\newblock {\em Discrete Math.}, 348(5):Article 114414, 2025.

\bibitem{Ste25}
R.~Steiner.
\newblock Fractional chromatic number vs. {H}all ratio.
\newblock {\em Combinatorica}, 45(4):Article 37, 2025.

\bibitem{Ste26}
R.~Steiner.
\newblock Locally bipartite subgraphs via multicolor {R}amsey numbers.
\newblock \href{https://arxiv.org/abs/2608.02522v2}{\tt arXiv:2608.02522v2},
  2026.

\bibitem{Ste-personal}
R.~Steiner.
\newblock Personal communication on 1st October 2026.

\bibitem{Tar18}
G.~Tardos.
\newblock On a graph coloring conjecture of {E}rd{\H{o}}s and {H}ajnal.
\newblock Talk at the ICM 2018 satellite meeting `Combinatorics: Extremal,
  Probabilistic and Additive'. Abstract at
  \url{https://epa-combinatorics2018.ime.usp.br/assets/programme.pdf}, 2018.
\newblock Retrieved 22nd September 2026.

\end{thebibliography}
	
\end{document}